\documentclass[letterpaper, 11pt,  reqno]{amsart}
\usepackage[margin=1.2in,marginparwidth=1.5cm, marginparsep=0.5cm]{geometry}

\usepackage{amsmath,amssymb,amscd,amsthm,amsxtra, esint, xcolor,mathtools,enumerate}
\usepackage{bbm}
\usepackage[implicit=true]{hyperref}
\allowdisplaybreaks[2]

\usepackage{color}

\newtheorem{theorem}{Theorem} [section]

\newtheorem{lemma}[theorem]{Lemma}
\newtheorem{proposition}[theorem]{Proposition}
\newtheorem{remark}[theorem]{Remark}

\newtheorem{corollary}[theorem]{Corollary}

\newtheorem*{ackno}{Acknowledgements}

\numberwithin{equation}{section}
\numberwithin{theorem}{section}

\makeatletter
\@namedef{subjclassname@2020}{\textup{2020} Mathematics Subject Classification}
\makeatother

\begin{document}
\baselineskip = 14pt

\title[Exact Temporal Variation for Fractional Stochastic Heat Equation]
{Exact Temporal Variation for Fractional Stochastic Heat Equation Driven by Space-Time White Noise}

\author[Y. Li H. Shu and L. Yan]
{Yongkang Li, Huisheng Shu and Litan Yan}

\thanks{Y. Li was supported by the Fundamental Research Funds for the Central Universities of China (CUSF-DH-T-2024066). H. Shu was supported by the National Natural Science Foundation of China (No. 62073071). L. Yan was supported by the National Natural Science Foundation of China (No. 11971101).}

\address{
Yongkang Li\\
Donghua University, 2999 North Renmin Rd., Songjiang, Shanghai 201620, P.R. China}

\email{yongkangli@mail.dhu.edu.cn}

\address{
Huisheng Shu\\
Donghua University, 2999 North Renmin Rd., Songjiang, Shanghai 201620, P.R. China
}

\email{hsshu@dhu.edu.cn}

\address{
Litan Yan\\
Donghua University, 2999 North Renmin Rd., Songjiang, Shanghai 201620, P.R. China
}

\email{litan-yan@hotmail.com}



\begin{abstract}  
In this paper, we consider the exact fractional variation for the temporal process of the solution to the fractional stochastic heat equation on $\mathbb{R}$ driven by a space-time white noise, and as an application we give the estimate of drift parameter.
\end{abstract}

\date{\today}
\maketitle
%

%
% \tableofcontents

\baselineskip = 14pt

\section{Introduction and Preliminaries}

In this paper, we consider the temporal variation of the solution $ \{u(t,x),t\geq0,x\in\mathbb{R}\} $ to the Cauchy problem for the one-dimensional fractional stochastic heat equation of the form
\begin{equation}\label{eq:equation1}
\left\{
\begin{aligned}
&\frac{\partial}{\partial t}u=\Delta_\alpha u+\sigma(u)\frac{\partial^2}{\partial t\partial x}W,\quad t\geq0,x\in\mathbb{R},\\
&u(0,x)=0,
\end{aligned}
\right.
\end{equation}
where $ W $ is a space-time white noise, $\Delta_\alpha=-(-\Delta)^{\alpha/2} $ is the fractional Laplacian operator with $ \alpha\in (1,2]$ and $u\mapsto \sigma(u)$ is a Lipschitz function. Our main aim is to prove the following theorem.
\begin{theorem}\label{th:temporal variation2}
Let $\alpha\in (1,2]$ and $ u(t,x) $ be the solution to (\ref{eq:equation1}). Consider fixed values $ 0< T_1<T_2 $ and $ x\in\mathbb{R} $. For $ i=0,1,2,\dots, n $, we define a partition of $[T_1,T_2]$ by $ t_i=T_1+i\delta $, where $ \delta=\frac{T_2-T_1}{n} $. Then the following limit holds in probability:
\begin{equation*}
\lim\limits_{n\rightarrow\infty}\frac{1}{n^{1/\alpha}}\sum_{i=1}^{n} (u(t_i,x)-u(t_{i-1},x))^2=c_\alpha(T_2-T_1)^{-\frac{1}{\alpha}} \int_{T_1}^{T_2}\sigma^2(u(r,x))dr,
\end{equation*}
where $c_\alpha=\dfrac{\Gamma(1/\alpha)}{\pi(\alpha-1)}$. Moreover, we have
\begin{multline}\label{eq:th_2}
\mathbb{E}\left|\frac{1}{n^{1/\alpha}}\sum_{i=1}^{n}(u(t_i,x)-u (t_{i-1},x))^2 -c_\alpha(T_2-T_1)^{-\frac{1}{\alpha}}\int_{T_1}^{T_2}\sigma^2(u(r,x))dr\right|\\
\leq C_{\alpha,T_1,{T_2}}n^{-\frac{(\alpha+1)(\alpha-1)}{2\alpha(3\alpha-1)}},
\end{multline}
with a positive constant $ C_{\alpha,T_1,T_2}$.
\end{theorem}

\begin{remark}
	The quadratic variation of a process $X(t)$ is defined as a limit in probability 
	\[[X,X](t)=\lim_{n\rightarrow\infty}\sum_{i=1}^{n}(X(t_i)-X(t_{i-1}))^2,\]
	where the limit is taken over an increasing sequence of partitions of $[0, t]$:
	\[0=t_0<t_1<\cdots <t_n=t,\]
	with $\delta_n=\max_{0\leq i\leq n}(t_i-t_{i-1})\rightarrow0.$
	
However, the corresponding limit does not exist as $n$ tends to infinity when applied to our solution. Indeed, taking the case with additive noise as an example, one notices that 
\begin{equation*}
	\sum_{i=1}^{n} (u(t_i,x)-u(t_{i-1},x))^2\rightarrow c_{\alpha,t}\,n^{\frac{1}{\alpha}}\quad \text{in probability}.
\end{equation*}
Motivated by Yan \textit{et al.}~\cite{yan2011generalized}, which introduces the concept of generalized quadratic covariation, i.e. 
\[\left[f\left(B^{H}\right), B^{H}\right]_{t}^{(H)}:=\lim _{\varepsilon \downarrow 0} \frac{1}{\varepsilon^{2 H}} \int_{0}^{t}\left\{f\left(B_{s+\varepsilon}^{H}\right)-f\left(B_{s}^{H}\right)\right\}\left(B_{s+\varepsilon}^{H}-B_{s}^{H}\right) d s^{2 H},\]
where $ f $ is a Borel function and $ B^H $ is a fractional Brownian motion with Hurst index $ 0<H<1/2 $, we call
\begin{equation*}
	\lim\limits_{n\rightarrow\infty}\frac{1}{n^{1/\alpha}}\sum_{i=1}^{n} (u(t_i,x)-u(t_{i-1},x))^2
\end{equation*}
the \textit{normalized quadratic variation} throughout the paper.
\end{remark}

As an application we estimate the drift parameter in the equation. When $\alpha=2$ and $\sigma(u)\equiv 1$, Swanson~\cite{Swanson} obtained the quartic variation of temporal process $t\mapsto u(t,x)$ with $\sigma(u)\equiv 1$.  When $\alpha=2$ and $\sigma(u)$ is a general Lipschitz function, Posp\'{\i}\v{s}il and Tribe~\cite{pospivsil2007parameter} considered the quartic variation of temporal process and quadratic variation of spatial process. When $1<\alpha<2$, Gamain and Tudor ~\cite{gamain2023exact} studied the limit behavior of the spatial quadratic variation of corresponding mild solution but no temporal variation. For further studies on variations of stochastic partial differential equations (SPDEs), we refer to  Cialenco and Huang~\cite{cialenco2020note}, Cui \emph{et al.}~\cite{cui2016temporal}, Torres \emph{et al.}~\cite{torres2014quadratic}, Tudor~\cite{tudor2022stochastic}, Zili and Zougar~~\cite{zili2020exact}, and the references therein.

Let $ G_\alpha(t,x) $ be the Green function of the operator $\Delta_\alpha $, i.e., the fundamental solution to the fractional heat equation. It is also the transition density function of the one-dimensional symmetric $\alpha$-stable L\'evy process which means $\int_\mathbb{R}G_\alpha(t,x)dx=1$. The solution of~(\ref{eq:equation1}) is understood in mild sense, that is,
\begin{equation}\label{eq:solution}
u(t,x)=\int_0^t\int_{\mathbb{R}}G_\alpha(t-r,x-z)\sigma(u(r,z))W(dr,dz).
\end{equation}
for $ t>0,x\in\mathbb{R} $, where the right-hand side is a stochastic integral with respect to the martingale measure and the existence and uniqueness of the mild solution have been shown (see, for example, Walsh ~\cite{Walsh1986}). It is important to note that the solution process $\{u(t,x),t\geq0,x\in\mathbb{R}\} $ is H\"older continuous of order less than $\frac{\alpha-1}{2\alpha}$ in $t$. For a more detailed analysis, we refer to Debbi and Dozzi~\cite{debbi2005solutions}, and Niu and Xie~\cite{niu2010regularity}.

The Green function $ G_\alpha(t,x)$ admits the following properties (see Wu~\cite{wu2011solution} and the detailed proof of (iv) to (vi) can be found in ~Debbi and Dozzi \cite{debbi2005solutions}):
\begin{enumerate}[(i)]
\item $ G_\alpha(t,x) $ is positive, real and symmetric with respect to the space variable $x\in\mathbb{R}$ for every $t>0$.
\item $ G_\alpha $ satisfies the semigroup property, that is, for $ 0<s<t $
\[G_\alpha(t+s,x)=\int_{\mathbb{R}}G_\alpha(t,y)G_\alpha(s,x-y)dy. \] 
\item For all $ x\in\mathbb{R} $ and $ t>0 $, we have
\[ G_\alpha(t,x)=t^{-1/\alpha}G_\alpha(1,xt^{-1/\alpha}), \]
which shows that the Green function enjoys the scaling property.
\item There exists a constant $ c_\alpha $ such that
\[ G_\alpha(1,x)\leq c_\alpha(1+|x|^{1+\alpha})^{-1}. \]
\item For $ 1<\gamma<\alpha+1 $, we have
\[ \int_{0}^{\infty}\int_{\mathbb{R}}|G_\alpha(1+s,z)-G_\alpha(s,z)|^\gamma dzds<\infty. \]
\item For $ \frac{\alpha+1}{2}<\gamma<\alpha+1 $, we have
\[ \int_{0}^{\infty}\int_{\mathbb{R}}|G_\alpha(s,1+z)-G_\alpha(s,z)|^\gamma dzds<\infty.\]
Particularly, we take $ \gamma=2 $, and it satisfies the above inequality because of $ \alpha\in(1,2] $, that is,
\begin{equation*}
\int_{0}^{\infty}\int_{\mathbb{R}}|G_\alpha(s,1+z)-G_\alpha(s,z)|^2 dzds<\infty.
\end{equation*}
\end{enumerate}

Recall that the Fourier transform of $ G_\alpha(t,x) $ can be expressed as 
\begin{equation}\label{eq:fourier trans}
\mathcal{F}G_\alpha(t,\cdot)(\xi)=\exp(-t|\xi|^\alpha).
\end{equation}

Moreover, the mild solution (\ref{eq:solution}) also satisfies the estimate for the $ p $-th moment of the full increment
\begin{equation}\label{eq:full increment}
\mathbb{E}[|u(t,x)-u(s,y)|^p]\leq c_{T,p,\alpha}\left(|t-s|^{\frac{\alpha-1}{2\alpha}p}+|x-y|^{\frac{\alpha-1}{2}p}\right)
\end{equation}
for every $ s,t\in[0,T], x, y\in\mathbb{R} $ and $ p\geq2 $, and the moment bounds
\begin{equation}\label{eq:p-moment}
\sup_{t\in[0,T],x\in\mathbb{R}}\mathbb{E}[|u(t,x)|^p]\leq c_{T,p,\alpha}
\end{equation}
for every $ p\geq1 $. (see again, Debbi and Dozzi \cite{debbi2005solutions}.)

For simplicity, throughout this paper, we denote a positive constant that depends solely on $ \alpha $ as $ C_\alpha $ and it should be noted that the value of this constant may vary in different contexts.

\section{Linear Case}
We first consider the case where $\sigma\equiv1$, in which case the corresponding mild solution to (\ref{eq:equation1}) can be written as
\[u_0(t,x)=\int_0^t\int_{\mathbb{R}}G_\alpha(t-s,x-z)W(ds,dz).\]
For all $ t>0 $ and $x\in\mathbb{R}^d$, we find that
\[ \mathbb{E}u_0^2(t,x)=c_{d,\alpha}\int_{0}^{t}u^{-d/\alpha}du, \]
and this integral is finite if and only if $ d<\alpha $, which is equivalent to $ d=1 $ given that $ \alpha\in(1,2] $. Thus, the existence of the mild solution is guaranteed only in spatial dimension $ d = 1 $.

In this section, we focus on analyzing the variation of the temporal process for a fixed $ x\in\mathbb{R} $ and a finite time interval $ [T_1, T_2] $. Let
$$ \Delta(t,\delta)=u_0(t+\delta,x)-u_0(t,x),\qquad \delta=\frac{T_2-T_1}{n},$$
denote the temporal increment associated with the uniform partition.
 Moreover, we investigate certain technical estimates associated with the mild solution. For $ s,t>0,x\in\mathbb{R} $, we have the following expression:
\begin{align}
\mathbb{E}u_0(t,x)u_0(s,x)&=\int_{0}^{s\wedge t}\int_{\mathbb{R}}G_\alpha(t-r,x-z)G_\alpha(s-r,x-z)dzdr\notag\\
&=\frac{c_\alpha}{2}
\left[
(t+s)^{\frac{\alpha-1}{\alpha}}
-
|t-s|^{\frac{\alpha-1}{\alpha}}
\right],\label{eq:cov}
\end{align}
where
\begin{equation}\label{eq:c-alpha}
	c_\alpha=\frac{\Gamma(1/\alpha)}{\pi(\alpha-1)}.
\end{equation}
It follows from the covariance formula that the centered Gaussian process
\(\{u_0(t,x),t\ge0\}\) has the same finite-dimensional distributions as
\[
\left(
\frac{c_\alpha}
{2^{1/\alpha}}
\right)^{1/2} B_t^{1/2,K},
\qquad K=\frac{\alpha-1}{\alpha},
\]
where \(\{B_t^{H,K},t\ge0\}\) is a bifractional Brownian motion.
Indeed, 
\[
\mathbb E[B_t^{H,K}B_s^{H,K}]
=
\frac{1}{2^K}
\left[
(t^{2H}+s^{2H})^K-|t-s|^{2HK}
\right].
\]
For \(H=1/2\) and \(K=(\alpha-1)/\alpha\), we get
\[
\begin{aligned}
	&\mathbb E\left[
	\left(\frac{c_\alpha}{2^{1/\alpha}}\right)^{1/2}B_t^{1/2,K}
	\left(\frac{c_\alpha}{2^{1/\alpha}}\right)^{1/2}B_s^{1/2,K}
	\right] \\
	&\quad =
	\frac{c_\alpha}{2^{1/\alpha}}
	\cdot
	\frac{1}{2^K}
	\left[
	(t+s)^K-|t-s|^K
	\right].
\end{aligned}
\]
Since \(K=(\alpha-1)/\alpha\), we have \(2^{-1/\alpha}2^{-K}=2^{-1}\). Therefore,
\[
\mathbb E\left[
\left(\frac{c_\alpha}{2^{1/\alpha}}\right)^{1/2}B_t^{1/2,K}
\left(\frac{c_\alpha}{2^{1/\alpha}}\right)^{1/2}B_s^{1/2,K}
\right]
=
\frac{c_\alpha}{2}
\left[
(t+s)^K-|t-s|^K
\right],
\]
which coincides with the covariance of \(u_0(\cdot,x)\). We refer to Houdr\'{e} and Villa \cite{houdre2003example} and Russo, and Tudor \cite{russo2006bifractional} for more details for the bifractional Brownian motion.

Consequently, we find that
\[
\mathbb{E}u_0^2(t,x)
=
\int_{0}^{t}\int_{\mathbb{R}}G_\alpha^2(t-r,x-z)\,dzdr
=
\frac{c_\alpha}{2^{1/\alpha}}t^{\frac{\alpha-1}{\alpha}},
\]
with \(c_\alpha\) given by \eqref{eq:c-alpha}.

Therefore, for \(0<s<t\), we have
\begin{align*}
	\mathbb{E}\left[(u_0(t,x)-u_0(s,x))^2\right]
	&=
	\mathbb{E}u_0^2(t,x)
	-2\mathbb{E}[u_0(t,x)u_0(s,x)]
	+\mathbb{E}u_0^2(s,x)\\
	&=
	c_\alpha\left[
	2^{-1/\alpha}
	\left(
	t^{\frac{\alpha-1}{\alpha}}
	+
	s^{\frac{\alpha-1}{\alpha}}
	\right)
	-
	\left(
	(t+s)^{\frac{\alpha-1}{\alpha}}
	-
	(t-s)^{\frac{\alpha-1}{\alpha}}
	\right)
	\right].
\end{align*}
In particular, taking \(t+\delta\) and \(t\) in place of \(t\) and \(s\), respectively, we obtain
\begin{align}
	\mathbb{E}\left[\Delta^2(t,\delta)\right]
	&=
	c_\alpha\left[
	2^{-1/\alpha}
	\left(
	(t+\delta)^{\frac{\alpha-1}{\alpha}}
	+
	t^{\frac{\alpha-1}{\alpha}}
	\right)
	-
	(2t+\delta)^{\frac{\alpha-1}{\alpha}}
	+
	\delta^{\frac{\alpha-1}{\alpha}}
	\right] \notag\\
	&=
	c_\alpha\delta^{\frac{\alpha-1}{\alpha}}
	+
	O(\delta^2),
	\qquad \delta\to0.
	\label{eq:torder2}
\end{align}
More precisely, there exists a constant \(C_{\alpha,t}>0\) such that, for all sufficiently small
\(\delta>0\),
\begin{equation}\label{eq:order2-error}
	\left|
	\mathbb E\left[\Delta^2(t,\delta)\right]
	-
	c_\alpha\delta^\frac{\alpha-1}{\alpha}
	\right|
	\le
	C_{\alpha,t}\,\delta^2.
\end{equation}

Since \(\Delta(t,\delta)\) is a centered Gaussian random variable, it follows that
\begin{equation}\label{eq:t4order}
	\mathbb{E}\left[\Delta^4(t,\delta)\right]
	=
	3c_\alpha^2\delta^{\frac{2\alpha-2}{\alpha}}
	+
	O\left(\delta^{\frac{3\alpha-1}{\alpha}}\right).
\end{equation}

\begin{lemma}\label{lem:temporal-off-diagonal}
	For every \(\alpha\in(1,2]\), there exists a constant
	\(C_{\alpha}>0\) such that, for all \(0\le s<t\) and \(0<\delta\le t-s\),
	\[
	\left|
	\mathbb E\left[\Delta(t,\delta)\Delta(s,\delta)\right]
	\right|
	\le
	C_{\alpha}\,\delta^2(t-s)^{\frac{\alpha-1}{\alpha}-2}.
	\]
\end{lemma}

\begin{proof}
	By \eqref{eq:cov}, we have
	\begin{align*}
		\mathbb E\left[\Delta(t,\delta)\Delta(s,\delta)\right]
		&=
		\mathbb E\left[
		(u_0(t+\delta,x)-u_0(t,x))
		(u_0(s+\delta,x)-u_0(s,x))
		\right] \\
		&=
		\frac{c_\alpha}{2}
		\Big[
		(t+s+2\delta)^{\frac{\alpha-1}{\alpha}}
		-2(t+s+\delta)^{\frac{\alpha-1}{\alpha}}
		+(t+s)^{\frac{\alpha-1}{\alpha}}
		\Big] \\
		&\qquad
		+
		\frac{c_\alpha}{2}\Big[
		(t-s+\delta)^{\frac{\alpha-1}{\alpha}}
		+(t-s-\delta)^{\frac{\alpha-1}{\alpha}}
		-2(t-s)^{\frac{\alpha-1}{\alpha}}
		\Big].
	\end{align*}
	
	We estimate the two brackets separately. 
	Taylor's formula gives
	\begin{align*}
		\left|
		(t+s+2\delta)^{\frac{\alpha-1}{\alpha}}
		-2(t+s+\delta)^{\frac{\alpha-1}{\alpha}}
		+(t+s)^{\frac{\alpha-1}{\alpha}}
		\right|
		&\le
		C_\alpha \delta^2 (t+s)^{\frac{\alpha-1}{\alpha}-2}\\
		&\leq C_\alpha \delta^2 (t-s)^{\frac{\alpha-1}{\alpha}-2},
	\end{align*}
	since $\frac{\alpha-1}{\alpha}-2<0$.
	
		For the second bracket, we consider two cases. If \(t-s\ge 2\delta\), the same estimate holds; that is
	\begin{equation*}
		\left|
		(t-s+\delta)^{\frac{\alpha-1}{\alpha}}
		+
		(t-s-\delta)^{\frac{\alpha-1}{\alpha}}
		-
		2(t-s)^{\frac{\alpha-1}{\alpha}}
		\right|\le
		C_\alpha \delta^2 (t-s)^{\frac{\alpha-1}{\alpha}-2}.
	\end{equation*}
	
	If \(\delta\le t-s<2\delta\), then
	\begin{equation*}
		\left|
		(t-s+\delta)^{\frac{\alpha-1}{\alpha}}
		+
		(t-s-\delta)^{\frac{\alpha-1}{\alpha}}
		-
		2(t-s)^{\frac{\alpha-1}{\alpha}}
		\right|\le
		C_\alpha (t-s)^{\frac{\alpha-1}{\alpha}}.
	\end{equation*}
	On the other hand, 
	\[
	\delta^2(t-s)^{\frac{\alpha-1}{\alpha}-2}
	=
	(t-s)^{\frac{\alpha-1}{\alpha}}
	\left(\frac{\delta}{t-s}\right)^2
	\ge
	\frac14 (t-s)^{\frac{\alpha-1}{\alpha}}.
	\]
	Thus,
	\[
	(t-s)^{\frac{\alpha-1}{\alpha}}
	\le
	4\delta^2(t-s)^{\frac{\alpha-1}{\alpha}-2}.
	\]
	Hence, in this case also,
	\begin{equation*}
		\left|
		(t-s+\delta)^{\frac{\alpha-1}{\alpha}}
		+
		(t-s-\delta)^{\frac{\alpha-1}{\alpha}}
		-
		2(t-s)^{\frac{\alpha-1}{\alpha}}
		\right|\le
		C_\alpha \delta^2 (t-s)^{\frac{\alpha-1}{\alpha}-2}.
	\end{equation*}
	
	Combining the estimates for the two brackets, we obtain
	\[
	\left|
	\mathbb E\left[\Delta(t,\delta)\Delta(s,\delta)\right]
	\right|
	\le
	C_\alpha \delta^2(t-s)^{\frac{\alpha-1}{\alpha}-2}.
	\]
	This completes the proof.
\end{proof}

\begin{theorem}\label{th:temporal variation}
Fix $ 0<T_1<T_2<\infty $ and $ x\in\mathbb{R} $. Define a time grid $ t_i=T_1+i\delta $, $ i=0,1,\dots,n $, where $ \delta=(T_2-T_1)/n $. Then the following limit holds in mean square:
\begin{equation*}
\lim\limits_{n\rightarrow\infty}\frac{1}{n^{1/\alpha}}\sum_{i=1}^{n}(u_0(t_i,x)-u_0(t_{i-1},x))^2=c_\alpha(T_2-T_1)^{\frac{\alpha-1}{\alpha}},
\end{equation*}
where $c_\alpha$ is given by \eqref{eq:c-alpha}.

Moreover, there exists a constant \(C_{\alpha,T_1}>0\) such that, for all sufficiently large \(n\),
\begin{multline*}
	\mathbb E\left|
	\frac{1}{n^{1/\alpha}}\sum_{i=1}^n 
	(u_0(t_i,x)-u_0(t_{i-1},x))^2
	-
	c_\alpha (T_2-T_1)^{\frac{\alpha-1}{\alpha}}
	\right|^2\\
	\le
	C_{\alpha,T_1}
	\left[
	(T_2-T_1)^{\frac{2\alpha-2}{\alpha}}n^{-1}
	+
	(T_2-T_1)^4 n^{-2-\frac2\alpha}
	\right].
\end{multline*}
\end{theorem}
\begin{proof}
		Recall that
		\[
		\Delta(t_{i-1},\delta)
		=
		u_0(t_i,x)-u_0(t_{i-1},x),
		\qquad i=1,\ldots,n.
		\]
		It is enough to prove that
		\[
		\frac{1}{n^{2/\alpha}}
		\sum_{i,j=1}^{n}
		\mathbb E\left[
		\left(
		\Delta^2(t_{i-1},\delta)
		-
		c_\alpha\delta^{\frac{\alpha-1}{\alpha}}
		\right)
		\left(
		\Delta^2(t_{j-1},\delta)
		-
		c_\alpha\delta^{\frac{\alpha-1}{\alpha}}
		\right)
		\right]
		\longrightarrow0, \quad \text{as} \quad n\to \infty.
		\]
		
		Since $\Delta(t_{i-1},\delta)$ and $\Delta(t_{j-1},\delta)$ are jointly centered Gaussian random variables, Isserlis' formula gives
		\begin{align*}
			&\mathbb E\left[
			\left(
			\Delta^2(t_{i-1},\delta)
			-
			c_\alpha\delta^{\frac{\alpha-1}{\alpha}}
			\right)
			\left(
			\Delta^2(t_{j-1},\delta)
			-
			c_\alpha\delta^{\frac{\alpha-1}{\alpha}}
			\right)
			\right]\\
			&\quad =
			\left(
			\mathbb E\Delta^2(t_{i-1},\delta)
			-
			c_\alpha\delta^{\frac{\alpha-1}{\alpha}}
			\right)
			\left(
			\mathbb E\Delta^2(t_{j-1},\delta)
			-
			c_\alpha\delta^{\frac{\alpha-1}{\alpha}}
			\right)\\
			&\qquad
			+2\left(
			\mathbb E[
			\Delta(t_{i-1},\delta)\Delta(t_{j-1},\delta)]
			\right)^2.
		\end{align*}
		Hence the desired second moment is bounded by the sum of the following two terms:
		\begin{align*}
			A_n
			&:=
			\frac{1}{n^{2/\alpha}}
			\sum_{i,j=1}^{n}
			\left|
			\mathbb E\Delta^2(t_{i-1},\delta)
			-
			c_\alpha\delta^{\frac{\alpha-1}{\alpha}}
			\right|
			\left|
			\mathbb E\Delta^2(t_{j-1},\delta)
			-
			c_\alpha\delta^{\frac{\alpha-1}{\alpha}}
			\right|,
		\end{align*}
		and
		\[
		B_n
		:=
		\frac{2}{n^{2/\alpha}}
		\sum_{i,j=1}^{n}
		\left(
		\mathbb E[
		\Delta(t_{i-1},\delta)\Delta(t_{j-1},\delta)]
		\right)^2.
		\]
		
		We first estimate \(A_n\). By the second-moment estimate, uniformly for
		\(t_{i-1}\in[T_1,T_2]\),
		\[
		\left|
		\mathbb E\Delta^2(t_{i-1},\delta)
		-
		c_\alpha\delta^{\frac{\alpha-1}{\alpha}}
		\right|
		\le
		C_{\alpha,T_1}\delta^2.
		\]
		Therefore,
		\[
		A_n
		\le
		\frac{C_{\alpha,T_1}}{n^{2/\alpha}}
		n^2\delta^4
		=
		C_{\alpha,T_1}(T_2-T_1)^4
		n^{-2-\frac2\alpha}.
		\]
		In particular, \(A_n\to0\) as $n\to \infty$.
		
		Now we estimate \(B_n\). Split \(B_n\) into the diagonal part and the off-diagonal part:
		\[
		B_n
		=
		\frac{2}{n^{2/\alpha}}
		\sum_{i=1}^{n}
		\left(
		\mathbb E\Delta^2(t_{i-1},\delta)
		\right)^2
		+
		\frac{4}{n^{2/\alpha}}
		\sum_{1\le j<i\le n}
		\left(
		\mathbb E[
		\Delta(t_{i-1},\delta)\Delta(t_{j-1},\delta)]
		\right)^2.
		\]
		
		By \eqref{eq:torder2} and the uniform version of \eqref{eq:order2-error},
			\[
			\mathbb E\Delta^2(t_{i-1},\delta)
			\le
			c_\alpha\delta^{\frac{\alpha-1}{\alpha}}
			+
			C_{\alpha,T_1}\delta^2.
			\]
			Since \(0<\delta\le1\) and \((\alpha-1)/\alpha<2\), we have
			\(\delta^2\le \delta^{(\alpha-1)/\alpha}\). Then	we get
		\[
		\frac{2}{n^{2/\alpha}}
		\sum_{i=1}^{n}
		\left(
		\mathbb E\Delta^2(t_{i-1},\delta)
		\right)^2
		\le
		C_{\alpha,T_1}
		n^{-\frac2\alpha}n\delta^{\frac{2\alpha-2}{\alpha}}=C_{\alpha,T_1}
		(T_2-T_1)^{\frac{2\alpha-2}{\alpha}}n^{-1},
		\]
		since \(\delta=(T_2-T_1)/n\).
		
		For the off-diagonal part, by Lemma \ref{lem:temporal-off-diagonal}, for \(i>j\),
		\[
		\left|
		\mathbb E[
		\Delta(t_{i-1},\delta)\Delta(t_{j-1},\delta)]
		\right|
		\le
		C_{\alpha}\,
		\delta^{\frac{\alpha-1}{\alpha}}
		|i-j|^{\frac{\alpha-1}{\alpha}-2}.
		\]
		Hence
		\begin{align*}
			&\frac{4}{n^{2/\alpha}}
			\sum_{1\le j<i\le n}
			\left(
			\mathbb E[
			\Delta(t_{i-1},\delta)\Delta(t_{j-1},\delta)]
			\right)^2\\
			&\le
			C_{\alpha}
			n^{-\frac2\alpha}
			\delta^{\frac{2\alpha-2}{\alpha}}
			\sum_{1\le j<i\le n}
			|i-j|^{\frac{2\alpha-2}{\alpha}-4}\\
			&= C_{\alpha}
			n^{-\frac2\alpha}
			\delta^{\frac{2\alpha-2}{\alpha}}\sum_{k=1}^{n-1}
			(n-k)k^{\frac{2\alpha-2}{\alpha}-4}\\
			&\leq C_{\alpha}
			n^{-\frac2\alpha}
			\delta^{\frac{2\alpha-2}{\alpha}}n\sum_{k=1}^{\infty}
			k^{\frac{2\alpha-2}{\alpha}-4}\\
			&\leq C_{\alpha}
			(T_2-T_1)^{\frac{2\alpha-2}{\alpha}}n^{-1},
		\end{align*}
	since 
	$	\sum_{k=1}^{\infty}
		k^{\frac{2\alpha-2}{\alpha}-4}<\infty.$

		Combining the estimates for \(A_n\) and \(B_n\), we obtain
		\[
		\mathbb E\left|
		\frac{1}{n^{1/\alpha}}\sum_{i=1}^{n}\Delta^2(t_{i-1},\delta)
		-
		c_\alpha(T_2-T_1)^{\frac{\alpha-1}{\alpha}}
		\right|^2
		\le
		C_{\alpha,T_1}
		\left[
		(T_2-T_1)^{\frac{2\alpha-2}{\alpha}}n^{-1}
		+
		(T_2-T_1)^4 n^{-2-\frac2\alpha}
		\right],
		\]
		which tends to zero as $n\to \infty$.
		 This proves the desired mean-square convergence.
\end{proof}
% % % % % % % % % % % %
% % % % % % % % % % % %
% % % % % % % % % % % %
\section{Non-linear Case}
In this section, we investigate the nonlinear fractional heat equation given by
  \begin{equation}
  \left\{
  \begin{aligned}
  &\frac{\partial}{\partial t}u=\Delta_\alpha u+\sigma(u)\frac{\partial^2}{\partial t\partial x}W,\quad t\geq0,x\in\mathbb{R}\\
  &u(0,x)=0,
  \end{aligned}
  \right.
  \end{equation}
where the coefficient $u\mapsto \sigma(u)$ is a Lipschitz function. In this section, we adopt a Fourier-analytic approach. The mild solution of the above equation can be expressed as
\begin{equation}\label{eq:solution2}
u(t,x)=\int_0^t\int_{\mathbb{R}}G_\alpha(t-r,x-z)\sigma(u(r,z))W(dr,dz) .
\end{equation}

Fix $ x\in\mathbb{R} $ and define the temporal increment as follows
\[ \bar{\Delta}(t,\delta)=u(t+\delta,x)-u(t,x) .\]
%Therefore, we approximate the temporal increment by $ \sigma(u(t(\delta),x))\widetilde{\Delta}(t,\delta) $, 
Let \(\widetilde W\) be an independent copy of \(W\). We define the modified temporal increment by
\begin{align*}
	\widetilde{\Delta}(t,\delta)
	&:=
	\int_0^{t(\delta)}\int_{\mathbb R}
	\left[
	G_\alpha(t+\delta-r,x-z)
	-
	G_\alpha(t-r,x-z)
	\right]\widetilde W(dr,dz)\\
	&\quad+
	\int_{t(\delta)}^{t+\delta}\int_{\mathbb R}
	\left[
	G_\alpha(t+\delta-r,x-z)
	-
	G_\alpha(t-r,x-z)
	\right]W(dr,dz),
\end{align*}
where $ t(\delta)=t-\delta^{\kappa} ,$ 
\begin{equation}\label{eq:kappa}
	\text{with}\quad\kappa=\frac{2\alpha}{3\alpha-1},
\end{equation}
for sufficiently small $\delta$, $t(\delta)\geq0$ uniformly for $t\in[T_1,T_2]$

%We need to define the "modified temporal increment": denote
%\begin{align*}
%\widetilde{\Delta}(t,\delta):=\int_0^{t(\delta)}&\int_{\mathbb{R}}(G_\alpha(t+\delta-r,x-z)-G_\alpha(t-r,x-z))\widetilde{W}(dr,dz)\\
%&+\int_{t(\delta)}^{t+\delta}\int_{\mathbb{R}}(G_\alpha(t+\delta-r,x-z)-G_\alpha(t-r,x-z))W(dr,dz)
%\end{align*}
%where $ t(\delta)=t-\delta^{\kappa} ,$ 
%\begin{equation}\label{eq:kappa}
%	\text{with}\quad\kappa=\frac{2\alpha}{3\alpha-1},
%\end{equation}
%for sufficiently small $\delta$, $t(\delta)\geq0$ uniformly for $t\in[T_1,T_2]$ with $ \widetilde{W}(dr,dz) $ representing an independent space-time white noise.
 Note that $G_\alpha(r,\cdot)=0$ for $r\leq 0$. When $\alpha=2$, the value of $\kappa$ reduces to $4/5$, which corresponds to the stochastic heat equation case in \cite{pospivsil2007parameter}. It is important to note that $ \widetilde{\Delta}(t,\delta) $ follows the same Gaussian distribution as the linear increment $ \Delta(t,\delta):=u_0(t+\delta,x)-u_0(t,x) $ and is independent of $\mathcal{F}_{t(\delta)}$,
where
\[\mathcal{F}_t = \sigma \{ W([0, r] \times A) : 0 \le r \le t, \ A \in \mathcal{B}_b(\mathbb{R}) \}.\]
%$ \sigma\{u(s,\cdot) : s\leq t(\delta)\} .$
% The value of $ t(\delta) $ is chosen to optimize the estimate in Theorem~\ref{th:temporal variation2}, which demonstrates the validity of the approximation.
 
The following result will be used repeatedly in the sequel. It provides an \(L^2(\Omega)\)-approximation of the nonlinear temporal increment by a frozen-coefficient modified Gaussian increment associated with the additive equation.
\begin{lemma}\label{le:lemma-estimation}
For any $ T>0 $, there exists $ c_{\alpha,T}<\infty $, such that
\begin{equation}\label{eq:le-esti}
\mathbb{E}[|\bar{\Delta}(t,\delta)-\sigma(u(t(\delta),x))\widetilde{\Delta}(t,\delta)|^2]\leq c_{\alpha,T}\delta^{\frac{4(\alpha-1)}{3\alpha-1}}
\end{equation}
holds for all $ x\in\mathbb{R}, 0\leq\delta^{\kappa}\leq t\wedge 1 $ and $ t\leq T. $
\end{lemma}
\begin{proof}
We split the difference into three parts:
\begin{align*}
&\qquad\bar{\Delta}(t,\delta)-\sigma(u(t(\delta),x))\widetilde{\Delta}(t,\delta)\\
&=\int_0^{t(\delta)} \int_{\mathbb{R}}\left(G_{\alpha}(t+\delta-r, x-z)-G_{\alpha}(t-r, x-z)\right) \sigma(u(r,z))W(dr,dz) \\
&\quad+\int_{t(\delta)}^{t+\delta} \int_{\mathbb{R}}\left(G_{\alpha}(t+\delta-r, x-z)-G_{\alpha}(t-r, x-z)\right) (\sigma(u(r,z))-\sigma(u(t(\delta),x))) W(dr,dz) \\
&\quad-\int_{0}^{t(\delta)} \int_{\mathbb{R}}\left(G_{\alpha}(t+\delta-r, x-z)-G_{\alpha}(t-r, x-z)\right) \sigma(u(t(\delta), x)) \widetilde{W}(dr,dz) \\
&:=I_1+I_2-I_3
\end{align*}
for all $x\in\mathbb{R}$, $0\leq\delta^{\kappa}\leq t\wedge 1$ and $t\leq T$. Consequently, we have
\begin{align*}
\mathbb{E}[|\bar{\Delta}(t,\delta)-\sigma(u(t(\delta),x))\widetilde{\Delta}(t,\delta)|^2]&=\mathbb{E}[|I_1+I_2-I_3|^2]\\
&\leq 4\mathbb{E}|I_1|^2+4\mathbb{E}|I_2|^2+2\mathbb{E}|I_3|^2.
\end{align*}
We now estimate each term in the above equation, individually. It follows from the Lipschitz condition of $ \sigma(u) $, It\^o isometry, (\ref{eq:p-moment}) and Parseval-Plancherel identity that
\begin{align*}
\mathbb{E}|I_1|^2&=\int_0^{t(\delta)} \int_{\mathbb{R}}\left(G_{\alpha}(t+\delta-r, x-z)-G_{\alpha}(t-r, x-z)\right)^2 \mathbb{E}[\sigma^2(u)]dzdr \\
&\leq \frac{c_{\alpha,T}}{2\pi}\int_0^{t(\delta)} \int_{\mathbb{R}}\left|\mathcal{F}G_{\alpha}(t+\delta-r, \cdot)(\xi)-\mathcal{F}G_{\alpha}(t-r, \cdot)(\xi)\right|^2d\xi dr.
\end{align*}
From (\ref{eq:fourier trans}), we know that 
\begin{eqnarray*}
\mathbb{E}|I_1|^2&\leq& \frac{c_{\alpha,T}}{2\pi}\int_0^{t(\delta)} \int_{\mathbb{R}}\left|e^{-(t+\delta-r)|\xi|^\alpha}-e^{-(t-r)|\xi|^\alpha}\right|^2\,d\xi\, dr\\
&=&\frac{c_{\alpha,T}}{2\pi}\int_0^{t(\delta)} \int_{\mathbb{R}}e^{-2(t-r)|\xi|^\alpha}(1-e^{-\delta|\xi|^\alpha})^2\,d\xi \,dr\\
&=&\frac{c_{\alpha,T}}{2\pi}\int_{\mathbb{R}}\frac{(1-e^{-\delta|\xi|^\alpha})^2}{2|\xi|^\alpha}e^{-2\delta^{\kappa}|\xi|^\alpha}\left(1-e^{-2|\xi|^\alpha\left(t-\delta^{\kappa}\right)}\right)\,d\xi.
\end{eqnarray*}
Then by using the fact that $1-e^{-x}\leq 1\wedge x$ for all $x\geq 0$ we derive that
\begin{equation*}
	\mathbb{E}|I_1|^2\leq \frac{c_{\alpha,T}}{2\pi}\int_0^\infty\delta^2\xi^\alpha e^{-2\delta^{\kappa}\xi^\alpha}\,d\xi\,=\,c_{\alpha,T}\delta^{2-\frac{\kappa(\alpha+1)}{\alpha}}\,=\,c_{\alpha,T}\delta^{\frac{4(\alpha-1)}{3\alpha-1}}
\end{equation*}
for all $x\in\mathbb{R},  0\leq\delta^{\kappa}\leq t\wedge 1$ and $t\leq T$.

The estimate for \(I_3\) follows in the same way, using the independence of
\(\widetilde W\) and \(\mathcal F_{t(\delta)}\), together with the moment bound
\(\sup_{s\le T,x\in\mathbb R}\mathbb E|\sigma(u(s,x))|^2<\infty\). Then we have
\begin{equation*}
\mathbb{E}|I_3|^2\leq c_{\alpha,T}\delta^{\frac{4(\alpha-1)}{3\alpha-1}},
\end{equation*}
for all $x\in\mathbb{R}, 0\leq\delta^{\kappa}\leq t\wedge 1$ and $t\leq T$. 

For the term $ \mathbb{E}|I_2|^2 $, it can be written as
\[
\int_{t(\delta)}^{t+\delta} \int_{\mathbb{R}}\left(G_{\alpha}(t+\delta-r, x-z)-G_{\alpha}(t-r, x-z)\right)^2 \mathbb{E}(\sigma(u)-\sigma(u(t(\delta),x)))^2dzdr .
\] 
It follows from (\ref{eq:full increment}) that
\begin{multline*}
\mathbb{E}|I_2|^2\leq c_{\alpha,T}\int_{t(\delta)}^{t+\delta}\int_{\mathbb{R}}\left(G_{\alpha}(t+\delta-r, x-z)-G_{\alpha}(t-r, x-z)\right)^2 \left(\delta^{\frac{\kappa(\alpha-1)}{\alpha}}+|x-z|^{\alpha-1}\right) dzdr\\
\leq c_{\alpha,T}\int_{t(\delta)}^{t+\delta}\int_{\mathbb{R}}G_{\alpha}(t+\delta-r, x-z)^2\left(\delta^{\frac{\kappa(\alpha-1)}{\alpha}}+|x-z|^{\alpha-1}\right) dzdr\\
+c_{\alpha,T}\int_{t(\delta)}^{t}\int_{\mathbb{R}}G_{\alpha}(t-r, x-z)^2\left(\delta^{\frac{\kappa(\alpha-1)}{\alpha}}+|x-z|^{\alpha-1}\right) dzdr
\end{multline*}
In fact, the estimations for these two terms are similar, and their results are identical; therefore, the proof of the latter is omitted. Hence, we write
\begin{align*}
	&\int_{t(\delta)}^{t+\delta}\int_{\mathbb{R}}G_{\alpha}(t+\delta-r, x-z)^2\left(\delta^{\frac{\kappa(\alpha-1)}{\alpha}}+|x-z|^{\alpha-1}\right) dzdr\\
	&\quad =\int_{t(\delta)}^{t+\delta}\int_{\mathbb{R}}G_\alpha^2(t+\delta-r,x-z)\cdot\delta^{\frac{\kappa(\alpha-1)}{\alpha}}dzdr\\
	&\qquad\qquad	+\int_{t(\delta)}^{t+\delta}\int_{\mathbb{R}}G_\alpha^2(t+\delta-r,x-z)|x-z|^{\alpha-1}dzdr\\
	&\quad\leq C_\alpha\delta^{\frac{2\kappa(\alpha-1)}{\alpha}}+\int_{t(\delta)}^{t+\delta}\int_{\mathbb{R}}G_\alpha^2(t+\delta-r,x-z)|x-z|^{\alpha-1}dzdr,
\end{align*}
where the second equality is derived easily by using the Parseval-Plancherel identity again.

To estimate the second term, recall the property (iv) of the kernel $ G_\alpha(t,x) $, we obtain
\begin{align*}
&\int_{t(\delta)}^{t+\delta}\int_{\mathbb{R}}G_\alpha^2(t+\delta-r,x-z)|x-z|^{\alpha-1}dzdr\\
&\leq c_\alpha\int_{t(\delta)}^{t+\delta}\int_{\mathbb{R}}(t+\delta-r)^{-\frac{2}{\alpha}}\left(1+|z(t+\delta-r)^{-\frac{1}{\alpha}}|^{1+\alpha}\right)^{-2}|z|^{\alpha-1}dzdr\\
&=c_\alpha\int_{t(\delta)}^{t+\delta}\int_{\mathbb{R}}(t+\delta-r)^{-\frac{2}{\alpha}+\frac{\alpha-1}{\alpha}+\frac{1}{\alpha}}(1+|z|^{1+\alpha})^{-2}|z|^{\alpha-1}dzdr\\
&=c_\alpha\int_{t(\delta)}^{t+\delta}(t+\delta-r)^{\frac{\alpha-2}{\alpha}}dr\leq c_\alpha\delta^{\frac{2\kappa(\alpha-1)}{\alpha}}.
\end{align*}
Therefore, we have
\[\mathbb{E}|I_2|^2\leq c_\alpha\delta^{\frac{4(\alpha-1)}{3\alpha-1}},\]
for all $x\in\mathbb{R}, 0\leq\delta^{\kappa}\leq t\wedge 1$ and $t\leq T$.

By combining the bounds on the terms $ \mathbb{E}|I_1|^2, \mathbb{E}|I_2|^2$ and $ \mathbb{E}|I_3|^2 $, we have thus proved the lemma.
\end{proof}
\begin{remark}
	It is interesting that our estimate expand the result in \cite{pospivsil2007parameter}, since the estimate (\ref{eq:le-esti}) reduces to $\delta^{4/5}$ when $\alpha=2$.
\end{remark}
We are ready to prove Theorem \ref{th:temporal variation2}.
\begin{proof}[Proof of Theorem \ref{th:temporal variation2}]
To prove this, it suffices to show that
\[ \frac{1}{n^{1/\alpha}}\sum_{i=1}^{n}(u(t_i,x)-u(t_{i-1},x))^2-c_\alpha(T_2-T_1)^{-\frac{1}{\alpha}}\int_{T_1}^{T_2}\sigma^2(u(r,x))dr\rightarrow0\quad \text{as}\quad n\rightarrow\infty .\]
Let \(\{\widetilde W^{(i)}\}_{i\ge1}\) be a sequence of mutually independent
space--time white noises, independent of \(W\). For each \(i=1,\ldots,n\), define
\(\widetilde\Delta_i(t_{i-1},\delta)\) in the same way as
\(\widetilde\Delta(t_{i-1},\delta)\), with \(\widetilde W\) replaced by
\(\widetilde W^{(i)}\). By construction,
\(\widetilde\Delta_i(t_{i-1},\delta)\) has the same distribution as the linear
increment \(\Delta(t_{i-1},\delta)\), and it is independent of
\(\mathcal F_{t_{i-1}(\delta)}\). Moreover, the estimate in Lemma
\ref{le:lemma-estimation} remains valid with
\(\widetilde\Delta(t_{i-1},\delta)\) replaced by
\(\widetilde\Delta_i(t_{i-1},\delta)\).  Hence we break the required convergence into three parts as follows
\begin{align}
&\frac{1}{n^{1/\alpha}}\sum_{i=1}^{n}\bar{\Delta}^2(t_{i-1},\delta)-c_\alpha(T_2-T_1)^{-\frac{1}{\alpha}}\int_{T_1}^{T_2}\sigma^2(u(r,x))dr\notag\\
&=\frac{1}{n^{1/\alpha}}\sum_{i=1}^{n}\left[\bar{\Delta}^2(t_{i-1},\delta)-\sigma^2(u(t_{i-1}(\delta),x))\widetilde{\Delta}_i^2(t_{i-1},\delta)\right]\label{eq:part11}\\
&\quad+\frac{1}{n^{1/\alpha}}\sum_{i=1}^{n}\sigma^2(u(t_{i-1}(\delta),x))(\widetilde{\Delta}_i^2(t_{i-1},\delta)-c_\alpha\delta^{\frac{\alpha-1}{\alpha}})\label{eq:part22}\\
&\quad+\frac{1}{n^{1/\alpha}}c_\alpha\delta^{\frac{\alpha-1}{\alpha}}\sum_{i=1}^{n}\sigma^2(u(t_{i-1}(\delta),x))-c_\alpha(T_2-T_1)^{-\frac{1}{\alpha}}\int_{T_1}^{T_2}\sigma^2(u(r,x))dr\label{eq:part33}\\
&:=J_1+J_2+J_3\notag
\end{align}
where $t_{i-1}(\delta)=t_{i-1}-\delta^{\kappa}$ and $\kappa$ is given by (\ref{eq:kappa}). 

%It is worth noting that the third part~(\ref{eq:part33}) converges almost surely to zero as $n\to \infty $ using the continuity of $t\mapsto u(t,x)$ for every $x\in {\mathbb R}$ in argument of Riemann sum. 
Indeed, for the third part~(\ref{eq:part33}),
\begin{equation}
J_3
=c_\alpha(T_2-T_1)^{-\frac{1}{\alpha}}\sum_{i=1}^{n}\int_{t_{i-1}}^{t_{i}}\left[\sigma^2(u(t_{i-1}(\delta),x))-\sigma^2(u(r, x))\right]dr.
\end{equation}

By the sample path continuity of $u$ and the Lipschitz continuity of $\sigma$, for almost every $\omega$, the map
\[s\mapsto\sigma^2(u(s,x;\omega))\]
is uniformly continuous on the compact interval $[T_1/2,T_2]$. Since
\[
\sup_{1\le i\le n}\sup_{r\in[t_{i-1},t_i]}
|t_{i-1}(\delta)-r|
\le
\delta^\kappa+\delta
\to0,
\]
then we have
\[|J_3|\leq c_\alpha(T_2-T_1)^{1-\frac{1}{\alpha}}\sup_{\substack{s,r\in[T_1/2,T_2]\\
		|s-r|\le \delta+\delta^{\kappa}}}
|\sigma^2(u(s,x))-\sigma^2(u(r,x))|\rightarrow0\]
almost surely as $n\rightarrow\infty$.

Now, let us deal with its convergence in $L^1(\Omega)$. By the Lipschitz condition on \(\sigma\),
the moment bound, Cauchy--Schwarz inequality, and the temporal regularity estimate (\ref{eq:full increment}), we get
\[
\begin{aligned}
	\mathbb E|J_{3}|
	&\le
	C_{\alpha,T_1,T_2}
	(T_2-T_1)^{-\frac1\alpha}
	\sum_{i=1}^{n}
	\int_{t_{i-1}}^{t_i}
	\left(
	\mathbb E|u(t_{i-1}(\delta),x)-u(r,x)|^2
	\right)^{1/2}dr\\
	&\le
	C_{\alpha,T_1,T_2}
	(T_2-T_1)^{-\frac1\alpha}
	\sum_{i=1}^{n}
	\int_{t_{i-1}}^{t_i}
	|t_{i-1}(\delta)-r|^{\frac{\alpha-1}{2\alpha}}dr\\
	&\le
	C_{\alpha,T_1,T_2}
	(T_2-T_1)^{-\frac1\alpha+1}
	\delta^{\frac{\kappa(\alpha-1)}{2\alpha}}\\
	&=
	C_{\alpha,T_1,T_2}
	(T_2-T_1)^{\frac{(\alpha-1)(4\alpha-1)}{\alpha(3\alpha-1)}
		}
	n^{-\frac{\alpha-1}{3\alpha-1}}.
\end{aligned}
\]
for all $x\in\mathbb{R}$ and $0< T_1<T_2$.

According to the Cauchy-Schwarz inequality and Lemma \ref{le:lemma-estimation}, we get that
\begin{align*}
\mathbb{E}|J_1|
&\le
\frac{1}{n^{1/\alpha}}
\sum_{i=1}^{n}
\mathbb E\left|
\bar{\Delta}^2(t_{i-1},\delta)
-
\sigma^2(u(t_{i-1}(\delta),x))
\widetilde{\Delta}_i^2(t_{i-1},\delta)
\right|\\
&\leq\frac{1}{n^{1/\alpha}}\sum_{i=1}^{n}\mathbb{E}[|\bar{\Delta}(t_{i-1},\delta)+\sigma(u(t_{i-1}(\delta),x))\widetilde{\Delta}_i(t_{i-1},\delta)|^2]^{\frac{1}{2}}\\
&\qquad\qquad\cdot\mathbb{E}[|\bar{\Delta}(t_{i-1},\delta)-\sigma(u(t_{i-1}(\delta),x))\widetilde{\Delta}_i(t_{i-1},\delta)|^2]^{\frac{1}{2}}\\
&\leq \frac{1}{n^{1/\alpha}}C_{\alpha,T_1,T_2}\delta^{\frac{\kappa(\alpha-1)}{\alpha}}\sum_{i=1}^{n}\mathbb{E}[|\bar{\Delta}(t_{i-1},\delta)+\sigma(u(t_{i-1}(\delta),x))\widetilde{\Delta}_i(t_{i-1},\delta)|^2]^{\frac{1}{2}}
\end{align*} 
for all $x\in\mathbb{R}$ and $0< T_1<T_2$.
Moreover,
\[
\begin{aligned}
	&\mathbb E\left|
	\bar{\Delta}(t_{i-1},\delta)
	+
	\sigma(u(t_{i-1}(\delta),x))
	\widetilde{\Delta}_i(t_{i-1},\delta)
	\right|^2\\
	&\le
	2\mathbb E|\bar{\Delta}(t_{i-1},\delta)|^2
	+
	2\mathbb E\left[
	\sigma^2(u(t_{i-1}(\delta),x))
	\widetilde{\Delta}_i^2(t_{i-1},\delta)
	\right]\le
	C_{\alpha,T_1,T_2}
	\delta^{\frac{\alpha-1}{\alpha}},
\end{aligned}
\]
where we used the temporal regularity of \(u\), the fact that
\(\widetilde{\Delta}_i(t_{i-1},\delta)\) has the same distribution as the
linear increment \(\Delta(t_{i-1},\delta)\), the moment bound for \(u\), and
the independence between \(\widetilde{\Delta}_i(t_{i-1},\delta)\) and
\(\mathcal F_{t_{i-1}(\delta)}\).
Therefore,
\begin{align*}
	\mathbb E|J_{1}|
	&\le
	C_{\alpha,T_1,T_2}
	n^{-1/\alpha}
	n\,
	\delta^{\frac{\kappa(\alpha-1)}{\alpha}}
	\delta^{\frac{\alpha-1}{2\alpha}}\\
	&=
	C_{\alpha,T_1,T_2}
	(T_2-T_1)^{
		\frac{(\alpha-1)(7\alpha-1)}{2\alpha(3\alpha-1)}}
	n^{-\frac{\alpha^2-1}{2\alpha(3\alpha-1)}}.
\end{align*}
In particular, \(J_{1}\to0\) in \(L^1(\Omega)\) as $n$ tends to infinity.

For the second term (\ref{eq:part22}), set
	\[
	\sigma_i:=\sigma(u(t_{i-1}(\delta),x)),
	\qquad
	Z_i:=\widetilde\Delta_i^2(t_{i-1},\delta)-c_\alpha\delta^{\frac{\alpha-1}{\alpha}}.
	\]
	We intend to prove that \(J_{2}\to0\) in \(L^2(\Omega)\). Indeed,
	\[
	\mathbb E|J_{2}|^2
	=
	\frac{1}{n^{2/\alpha}}
	\sum_{i,j=1}^n
	\mathbb E[\sigma_i^2\sigma_j^2Z_iZ_j].
	\]
	
	We first record some elementary estimates. Since \(\widetilde\Delta_i\) has the same
	distribution as the linear increment \(\Delta(t_{i-1},\delta)\), the estimates in the
	linear case imply
	\[
	\mathbb E Z_i^2\le C_{\alpha,T_1}\delta^{\frac{2(\alpha-1)}{\alpha}},
	\qquad
	|\mathbb E Z_i|\le C_{\alpha,T_1}\delta^2.
	\]

	Denote
	$m_n:=\left\lfloor \sqrt{n}\right\rfloor$.
	We split the double sum into the near-diagonal part \(|i-j|\leq m_n\) and the far-off-diagonal
	part \(|i-j|> m_n\).
	
	For the near-diagonal part, by Hölder's inequality,
	\[
	\left|
	\mathbb E[\sigma_i^2\sigma_j^2Z_iZ_j]
	\right|
	\le
	\|\sigma_i^2\|_4
	\|\sigma_j^2\|_4
	\|Z_i\|_4\|Z_j\|_4.
	\]
	By the fact that for any $a,b\in\mathbb{R}$, 
	$|a-b|^4\leq C(|a|^4+|b|^4)$. Then, by using that $\widetilde{\Delta}_i(t_{i-1},\delta)$ is centered Gaussian and has the same distribution as $\Delta(t_{i-1},\delta)$, we have 
	\[\mathbb{E}|Z_i|^4\leq C\left(\mathbb{E}\left|\widetilde{\Delta}_i(t_{i-1},\delta)\right|^8+\delta^{\frac{4(\alpha-1)}{\alpha}}\right)\leq C\delta^{\frac{4(\alpha-1)}{\alpha}},\]
	which implies $\|Z_i\|_4\leq C\delta^{\frac{\alpha-1}{\alpha}} $. Hence, together with the moment bound for the $u$ and the Lipschitz continuity of \(\sigma\), we obtain
	\[
	\frac{1}{n^{2/\alpha}}
	\sum_{\substack{1\le i,j\le n\\ |i-j|\leq m_n}}
	\left|
	\mathbb E[\sigma_i^2\sigma_j^2Z_iZ_j]
	\right|
	\le
	C_{\alpha,T_1,T_2}
	n^{-\frac2\alpha}n\,m_n\,\delta^{\frac{2(\alpha-1)}{\alpha}}
	\le
	C_{\alpha,T_1,T_2}n^{-\frac12}.
	\]
	
	Now consider the far-off-diagonal part. Without loss of generality, assume \(i>j\) and
	\(i-j> m_n\), then
	$t_j\le t_{i-1}(\delta)$.
	Hence \(\sigma_i^2\sigma_j^2Z_j\) is measurable with respect to
	\[
	\mathcal H_{i,j}:=
	\mathcal F_{t_{i-1}(\delta)}\vee\sigma\left(\widetilde W^{(j)}\right),
	\]
	whereas \(Z_i\) is independent of \(\mathcal H_{i,j}\).
	Consequently,
	\[
	\mathbb E[\sigma_i^2\sigma_j^2Z_iZ_j]
	=
	\mathbb E[\sigma_i^2\sigma_j^2Z_j]\mathbb E[Z_i].
	\]
	Therefore, by Hölder's inequality,
	\[
	\frac{1}{n^{2/\alpha}}
	\sum_{\substack{1\le j<i\le n\\ i-j> m_n}}
	\left|
	\mathbb E[\sigma_i^2\sigma_j^2Z_iZ_j]
	\right|
	\le
	C_{\alpha,T_1,T_2}
	n^{-\frac2\alpha}n^2\delta^{\frac{\alpha-1}{\alpha}+2}
	\le
	C_{\alpha,T_1,T_2}
	n^{-\frac{\alpha+1}{\alpha}}.
	\]
	
	Combining the near- and far-off-diagonal estimates, we get
	\[
	\mathbb E|J_{2}|^2
	\le
	C_{\alpha,T_1,T_2}
	\left(
	n^{-\frac12}
	+
	n^{-\frac{\alpha+1}{\alpha}}
	\right)
	\to0.
	\]
	Thus \(J_{2,n}\to0\) in \(L^2(\Omega)\). By Cauchy--Schwarz,
	\[
	\mathbb E|J_2|
	\le
	(\mathbb E|J_2|^2)^{1/2}
	\le
	C_{\alpha,T_1,T_2}
	\left(
	n^{-\frac14}
	+
	n^{-\frac{\alpha+1}{2\alpha}}
	\right)
	\le
	C_{\alpha,T_1,T_2}n^{-\frac14}.
	\]
	Combining the estimates for \(J_1\), \(J_2\), and \(J_3\), we obtain
	\[
	\begin{aligned}
		\mathbb E|J_1+J_2+J_3|
		&\le
		C_{\alpha,T_1,T_2}
		\left(
		n^{-\frac{\alpha^2-1}{2\alpha(3\alpha-1)}}
		+n^{-1/4}
		+n^{-\frac{\alpha-1}{3\alpha-1}}
		\right)  \\
		&\le
		C_{\alpha,T_1,T_2}
		n^{-\frac{(\alpha+1)(\alpha-1)}{2\alpha(3\alpha-1)}},
	\end{aligned}
	\]
	since $\alpha\in(1,2]$.
	This proves \eqref{eq:th_2}. The convergence in probability follows immediately from \eqref{eq:th_2}. The proof is complete.
\end{proof}
\begin{remark}\label{rem:uniform-x}
	Although Theorem \ref{th:temporal variation2} is stated for a fixed spatial point
	\(x\in\mathbb R\), the estimate \eqref{eq:th_2} is uniform for \(x\) in any compact
	subset of \(\mathbb R\). More precisely, for every compact set \(K\subset\mathbb R\),
	there exists a constant \(C_{\alpha,T_1,T_2,K}>0\) such that
	\begin{multline*}
		\sup_{x\in K}
		\mathbb{E}\left|
		\frac{1}{n^{1/\alpha}}\sum_{i=1}^{n}(u(t_i,x)-u(t_{i-1},x))^2
		-
		c_\alpha(T_2-T_1)^{-\frac1\alpha}
		\int_{T_1}^{T_2}\sigma^2(u(r,x))\,dr
		\right|\\
		\le
		C_{\alpha,T_1,T_2,K}
		n^{-\frac{(\alpha+1)(\alpha-1)}{2\alpha(3\alpha-1)}}.
	\end{multline*}
	Indeed, all the estimates used in the proof, including the moment bounds,
	the temporal regularity estimate, the approximation estimate in Lemma
	\ref{le:lemma-estimation}, and the estimates for the linear increments, can be
	chosen uniformly for \(x\in K\). In the present translation-invariant setting, the
	constants are in fact independent of \(x\).
\end{remark}

Motivated by Hildebrandt and Trabs~\cite{Hildebrandt2021}, we will introduce the average of the temporal quadratic variations at different places. From Theorem \ref{th:temporal variation2}, we deduce its asymptotic behavior as follows.

\begin{corollary}\label{cor:space-time-average}
		Let \(0<T_1<T_2\), and let \(t_i=T_1+i\delta\), where
		\(\delta=(T_2-T_1)/n\). For \(m,n\ge1\), define
		\[
		V_{n,m}(u)
		:=
		\frac{1}{m}\sum_{j=0}^{m-1}
		\frac{1}{n^{1/\alpha}}\sum_{i=1}^{n}
		\bigl(u(t_i,x_j)-u(t_{i-1},x_j)\bigr)^2,
		\]
		where
		\[
		x_j=\frac{j}{m},\qquad j=0,1,\ldots,m.
		\]
		Then, as \(n,m\to\infty\),
		\[
		V_{n,m}(u)
		\longrightarrow
		c_\alpha(T_2-T_1)^{-\frac1\alpha}
		\int_0^1\int_{T_1}^{T_2}\sigma^2(u(t,x))\,dt\,dx
		\]
		in \(L^1(\Omega)\). Moreover, for all sufficiently large \(n\) and \(m\),
		\[
		\mathbb E\left|
		V_{n,m}(u)
		-
		c_\alpha(T_2-T_1)^{-\frac1\alpha}
		\int_0^1\int_{T_1}^{T_2}\sigma^2(u(t,x))\,dt\,dx
		\right|
		\le
		C_{\alpha,T_1,T_2}
		\left(
		n^{-\frac{(\alpha+1)(\alpha-1)}{2\alpha(3\alpha-1)}}
		+
		m^{-\frac{\alpha-1}{2}}
		\right).
		\]
	\end{corollary}
	
	\begin{proof}
		We decompose the difference as
		\[
		\begin{aligned}
			&V_{n,m}(u)
			-
			c_\alpha(T_2-T_1)^{-\frac1\alpha}
			\int_0^1\int_{T_1}^{T_2}\sigma^2(u(t,x))\,dt\,dx \\
			&=
			V_{n,m}(u)
			-
			\frac{1}{m}\sum_{j=0}^{m-1}
			c_\alpha(T_2-T_1)^{-\frac1\alpha}
			\int_{T_1}^{T_2}\sigma^2(u(t,x_j))\,dt \\
			&\quad+
			\frac{1}{m}\sum_{j=0}^{m-1}
			c_\alpha(T_2-T_1)^{-\frac1\alpha}
			\int_{T_1}^{T_2}\sigma^2(u(t,x_j))\,dt
			-
			c_\alpha(T_2-T_1)^{-\frac1\alpha}
			\int_0^1\int_{T_1}^{T_2}\sigma^2(u(t,x))\,dt\,dx \\
			&=:D_1+D_2.
		\end{aligned}
		\]
		
		For \(D_1\), by Remark \ref{rem:uniform-x} applied with \(K=[0,1]\), we immediately obtain
		\[
			\mathbb E|D_1|\le
			C_{\alpha,T_1,T_2}\,
			n^{-\frac{(\alpha+1)(\alpha-1)}{2\alpha(3\alpha-1)}}.
		\]
		
		For \(D_2\), we get
		\[
		\begin{aligned}
			\mathbb E|D_2|
			&\le
			C_{\alpha,T_1,T_2}
			\sum_{j=0}^{m-1}
			\int_{x_j}^{x_{j+1}}\int_{T_1}^{T_2}
			\mathbb E\left|
			\sigma^2(u(t,x_j))-\sigma^2(u(t,x))
			\right|\,dt\,dx.
		\end{aligned}
		\]
		Hence, by the Lipschitz continuity for $\sigma$, moment bound for \(u\), Cauchy--Schwarz inequality, and the spatial
		regularity estimate,
		\[
		\mathbb E\left|
		\sigma^2(u(t,x_j))-\sigma^2(u(t,x))
		\right|
		\le
		C_{\alpha,T_1,T_2}
		\left(
		\mathbb E|u(t,x_j)-u(t,x)|^2
		\right)^{1/2}
		\le
		C_{\alpha,T_1,T_2}|x_j-x|^{\frac{\alpha-1}{2}}.
		\]
		Therefore,
		\[
		\begin{aligned}
			\mathbb E|D_2|
			&\le
			C_{\alpha,T_1,T_2}
			\sum_{j=0}^{m-1}
			\int_{x_j}^{x_{j+1}}
			|x_j-x|^{\frac{\alpha-1}{2}}\,dx  \\
			&\le
			C_{\alpha,T_1,T_2}
			m\left(\frac1m\right)^{\frac{\alpha+1}{2}}
			=
			C_{\alpha,T_1,T_2}
			m^{-\frac{\alpha-1}{2}}.
		\end{aligned}
		\]
		
		Combining the estimates for \(D_1\) and \(D_2\), we obtain
		\[
		\mathbb E\left|
		V_{n,m}(u)
		-
		c_\alpha(T_2-T_1)^{-\frac1\alpha}
		\int_0^1\int_{T_1}^{T_2}\sigma^2(u(t,x))\,dt\,dx
		\right|
		\le
		C_{\alpha,T_1,T_2}
		\left(
		n^{-\frac{(\alpha+1)(\alpha-1)}{2\alpha(3\alpha-1)}}
		+
		m^{-\frac{\alpha-1}{2}}
		\right).
		\]
		The \(L^1(\Omega)\)-convergence follows immediately.
\end{proof}

\section{Application}
In this section, we present an application for drift parameter estimation associated with the results established in two sections.

We first consider the linear case. Let $\{u_\sigma(t,x),t\geq0,x\in\mathbb{R}\}$ be solution of the fractional stochastic heat equation
\begin{equation}
\left\{
\begin{aligned}
&\frac{\partial}{\partial t}u_\sigma=\Delta_\alpha u_\sigma+\sigma\frac{\partial^2}{\partial t\partial x}W,\quad t\geq0,x\in\mathbb{R}\\
&u_\sigma(0,x)=0,
\end{aligned}
\right.
\end{equation}
where $ \sigma $ is a real positive constant parameter and $\Delta_\alpha:=-(-\Delta)^{\frac\alpha2}$ is Laplacian operator of order $\alpha\in(1,2]$. The mild solution can be expressed as
\[  
u_\sigma(t,x)=\sigma\int_0^t\int_{\mathbb{R}}G_\alpha(t-r,x-z)W(dr,dz).
\] 
Fix \(0<T_1<T_2\), \(x\in\mathbb R\), and let
\(t_i=T_1+i\delta\), where \(\delta=(T_2-T_1)/n\), Theorem \ref{th:temporal variation} implies that
\[   \lim\limits_{n\rightarrow\infty}\frac{1}{n^{1/\alpha}}\sum_{i=1}^{n}(u_\sigma(t_i,x)-u_\sigma(t_{i-1},x))^2 =\sigma^2c_\alpha(T_2-T_1)^{\frac{\alpha-1}{\alpha}} 
\] 
in probability where $ c_\alpha=\dfrac{\Gamma(1/\alpha)}{\pi(\alpha-1)}$. Define
\[  
\hat{\sigma}_n^2=\frac{\sum_{i=1}^{n}(u_\sigma(t_i,x)-u_\sigma(t_{i-1},x))^2}{n^{1/\alpha} c_\alpha(T_2-T_1)^{\frac{\alpha-1}{\alpha}}}.
\] 
Then \(\hat\sigma_n^2\to\sigma^2\) in probability. Since \(\sigma>0\), if we define
\(\hat\sigma_n=(\hat\sigma_n^2)^{1/2}\), then \(\hat\sigma_n\to\sigma\) in probability.

We now consider the nonlinear case. Let $ \{u_\mu(t,x):t\geq0, x\in\mathbb{R}\} $ be the solution of the fractional stochastic heat equation
 \begin{equation}\label{eq:the4eq}
  \left\{
  \begin{aligned}
  &\frac{\partial}{\partial t}u_\mu=\mu\Delta_\alpha u_\mu+\sigma(u_\mu(t,x))\frac{\partial^2}{\partial t\partial x}W,\quad t\geq0,x\in\mathbb{R}\\
  &u_\mu(0,x)=0,
  \end{aligned}
  \right.
 \end{equation}
where $ \mu>0 $ is a real constant parameter. Then the mild solution can be written as
\begin{equation}\label{eq:solution1}
u_\mu(t,x)=\int_0^t\int_{\mathbb{R}}G_\alpha(\mu(t-s),x-z)\sigma(u_\mu(s,z))W(ds,dz).
\end{equation}
We simplify this equation by employing linear scaling, which can be expressed by using the following lemma.
\begin{lemma}\label{le:lemma}
Let $ \{u_\mu(t,x),t\geq0,x\in\mathbb{R}\} $ be the mild solution to (\ref{eq:the4eq}), and let $ v_\mu(t,x)=u_\mu(t/\mu,x) $ for all $ t\geq0, x\in\mathbb{R}. $ Then, the process $ \{v_\mu(t,x),t\geq0, x\in\mathbb{R}\} $ satisfies
\begin{equation}
\left\{
\begin{aligned}
&\frac{\partial}{\partial t}v_\mu=\Delta_\alpha v_\mu+\mu^{-\frac{1}{2}}\sigma(v_\mu(t,x))\frac{\partial^2}{\partial t\partial x}\widetilde{W},\quad t\geq0,x\in\mathbb{R}\\
&v_\mu(0,x)=0,
\end{aligned}
\right.
\end{equation}
where $ \widetilde{W}(dt,dx) $ is a space-time white noise.
 \end{lemma}
 \begin{proof}
 According to (\ref{eq:the4eq}), we have, for every $ t\geq 0, x\in\mathbb{R} $,
 \begin{align*}
 v_\mu(t,x)&=u_\mu\left(\frac{t}{\mu},x\right)=\int_0^{\frac{t}{\mu}}\int_{\mathbb{R}}G_\alpha(t-\mu s,x-z)\sigma(u_\mu(s,z))W(ds,dz)\\
 &=\int_0^t\int_{\mathbb{R}}G_\alpha(t-s,x-z)\sigma(u_\mu(s/\mu,z))W\left(d\frac{s}{\mu},dz\right)\\
 &=\mu^{-\frac{1}{2}}\int_0^t\int_{\mathbb{R}}G_\alpha(t-s,x-z)\sigma(v_\mu(s,z))\widetilde{W}(ds,dz).
 \end{align*}
for all $t\geq0$ and $x\in\mathbb{R}$. Note that $ \widetilde{W} $ is also a space-time white noise which follows from the $ \frac{1}{2}$-self-similarity of the space-time white noise in time, that is
 \begin{equation*}
 \widetilde{W}([s,t]\times[a,b])=\sqrt{\mu}W\left(\left[\frac{s}{\mu},\frac{t}{\mu} \right]\times[a,b]\right) ,
 \end{equation*}
which follows from the scaling property of space–time white noise under deterministic time change, see e.g~Walsh \cite{Walsh1986}.
 \end{proof}
As an immediate result, we have the following proposition.

\begin{proposition}
		Fix $0<T_1<T_2$ and $x\in\mathbb{R}$. For $i=1,\dots,n$, define a time grid by
		\[
		t_i=T_1+i\delta, \quad \delta=\frac{T_2-T_1}{n}.
		\]
		Let $\{u_\mu(t,x): t\geq 0, x\in\mathbb{R}\}$ be the mild solution to (\ref{eq:the4eq}).
		Then the limit
		\begin{equation}
			\lim_{n\rightarrow\infty}\frac{1}{n^{1/\alpha}}
			\sum_{i=1}^{n}(u_\mu(t_i,x)-u_\mu(t_{i-1},x))^2
			=
			\frac{c_\alpha(T_2-T_1)^{-\frac{1}{\alpha}}}{\mu^{\frac{1}{\alpha}}}
			\int_{T_1}^{T_2}\sigma^2(u_\mu(r,x))dr
		\end{equation}
		holds in probability, where $ c_\alpha $ is given by \eqref{eq:c-alpha}.
		
		Moreover, if we suppose that $ |\sigma(y)|\geq\sigma_0>0 $ for all $ y\in\mathbb{R} $ and define
		\begin{equation}
			\hat{\mu}_n=
			\left[
			\frac{c_\alpha(T_2-T_1)^{\frac{\alpha-1}{\alpha}}}{n^{(\alpha-1)/\alpha}}
			\frac{\sum_{i=1}^{n}\sigma^2(u_\mu(t_i,x))}
			{ \sum_{i=1}^{n}(u_\mu(t_i,x)-u_\mu(t_{i-1},x))^2}
			\right]^\alpha,
		\end{equation}
		then $\hat{\mu}_n \to \mu$ in probability as $n\to \infty$ and $\hat{\mu}_n$ is a consistent estimator of $\mu$.
	\end{proposition}
	
	\begin{proof}
		Recall that $v_\mu(t,x)=u_\mu(t/\mu,x)$.  
		Applying Theorem~\ref{th:temporal variation2} and Lemma~\ref{le:lemma}, we have
		\[
		\lim_{n\rightarrow\infty}
		\frac{1}{n^{1/\alpha}}
		\sum_{i=1}^{n}(v_\mu(s_i,x)-v_\mu(s_{i-1},x))^2
		=
		\frac{c_\alpha(S_2-S_1)^{-\frac{1}{\alpha}}}{\mu}
		\int_{S_1}^{S_2}\sigma^2(v_\mu(r,x))dr,
		\]
		where we define $s_i=\mu t_i$, $S_1=\mu T_1$ and $S_2=\mu T_2$.
		
		Since
		\[
		v_\mu(s_i,x)-v_\mu(s_{i-1},x)
		= u_\mu(t_i,x) - u_\mu(t_{i-1},x),
		\]
		we obtain
		\begin{align*}
			\lim_{n\rightarrow\infty}
			\frac{1}{n^{1/\alpha}}
			\sum_{i=1}^{n}( u_\mu(t_i,x)- u_\mu(t_{i-1},x))^2
			&=
			\frac{c_\alpha (\mu T_2-\mu T_1)^{-1/\alpha}}{\mu}
			\int_{\mu T_1}^{\mu T_2}\sigma^2(v_\mu(r,x))dr\\
			&=
			\frac{c_\alpha (T_2-T_1)^{-1/\alpha}}{\mu^{1/\alpha}}
			\int_{T_1}^{T_2}\sigma^2( u_\mu(r,x))dr,
		\end{align*}
		by the change of variables $r=\mu s$.
		
		Now, using the Riemann sum approximation
		\[
		\int_{T_1}^{T_2}\sigma^2( u_\mu(r,x))dr
		=\lim_{n\to\infty}\sum_{i=1}^{n}\sigma^2( u_\mu(t_i,x))(t_i-t_{i-1}),
		\]
		we conclude that
		\[
		\hat{\mu}_n \to \mu \quad \text{in probability as } n\to \infty,
		\]
		which completes the proof.
	\end{proof}

The same argument can be combined with the averaged temporal variation result
to construct a spatially averaged consistent estimator for \(\mu\). We omit the
details.

\begin{ackno}
	The authors thank the anonymous referee for constructive advice that improved the paper.
\end{ackno}

\end{document}